\documentclass[12pt]{article}
\usepackage{amssymb}
\usepackage{amsthm}
\usepackage{amsmath}
\usepackage{graphicx}
\usepackage{enumerate}

\DeclareMathOperator{\Max}{Max}

\DeclareRobustCommand*{\tautequiv}{\Relbar\joinrel\mathrel|\mathrel|\joinrel\Relbar}

\newtheorem{theorem}{Theorem}[section]
\newtheorem{definition}[theorem]{Definition}
\newtheorem{lemma}[theorem]{Lemma}
\newtheorem{proposition}[theorem]{Proposition}

\title{The logic induced by effect algebras}
\author{Ivan~Chajda, Radom\'ir~Hala\v s and Helmut~L\"anger}
\date{}
\begin{document}
\footnotetext[1]{Support of the research by \"OAD, project CZ~02/2019, support of the research of the first and second author by IGA, project P\v rF~2020~014, and support of the research of the first and third author by the Austrian Science Fund (FWF), project I~4579-N, and the Czech Science Foundation (GA\v CR), project 20-09869L, is gratefully acknowledged.}
\maketitle
\begin{abstract}
Effect algebras form an algebraic formalization of the logic of quantum mechanics. For lattice effect algebras $\mathbf E$ we investigate a natural implication and prove that the implication reduct of $\mathbf E$ is term equivalent to $\mathbf E$. Then we present a simple axiom system in Gentzen style in order to axiomatize the logic induced by lattice effect algebras. For effect algebras which need not be lattice-ordered we introduce a certain kind of implication which is everywhere defined but whose result need not be a single element. Then we study effect implication algebras and prove the correspondence between these algebras and effect algebras satisfying the Ascending Chain Condition. We present an axiom system in Gentzen style also for not necessarily lattice-ordered effect algebras and prove that it is an algebraic semantics for the logic induced by finite effect algebras.
\end{abstract}

{\bf AMS Subject Classification:} 03G12, 03G25, 06A11, 06F99

{\bf Keywords:} Lattice effect algebra, lattice effect implication algebra, effect algebra, effect implication algebra, finite effect algebra, Gentzen system, algebraic semantics

\section{Introduction}

Effect algebras were introduced by D.~Foulis and M.~K.~Bennett (\cite{FB}, see also \cite{BF} and \cite{BH}) as an algebraic axiomatization of the logic of quantum mechanics. For enabling deductions and derivations in this logic, it is necessary to introduce the connective implication. The problem is that though the binary operation of an effect algebra is only partial, implication should be defined everywhere. If the considered effect algebra is lattice-ordered then implication is usually defined by $x\rightarrow y:=x'+(x\wedge y)$ and called {\em Sasaki implication}, see e.g.\ \cite{BDS}, \cite{CLa}, \cite{FP} and \cite{RSS}. Properties of this implication were described in these papers and a certain axiomatization of the corresponding logic in Gentzen style (see e.g.\ \cite{BP}) was derived in \cite{RSS}. However, non-lattice-ordered effect algebras are more important then lattice-ordered ones. As shown in \cite{CHK}, any effect algebra can be completed to a {\em basic algebra} which is total. For commutative basic algebras a Gentzen system was presented in \cite{BH} and for basic algebras in \cite C. This motivates us to find such an axiomatization also for not necessarily lattice-ordered effect algebras.

Another reason for introducing implication in effect algebras is to show that this connective is related to conjunction via (left) adjointness and hence effect algebras can be considered as left residuated structures, see \cite{CLa} and \cite{CLb}.

The paper is organized as follows: First we introduce the {\em natural implication} instead of the Sasaki implication. Then we define lattice effect implication algebras having this type of implication together with one constant as fundamental operations. We prove that lattice effect algebras and lattice effect implication algebras are term equivalent and in a natural one-to-one correspondence. Then we derive an algebraic semantics for lattice effect implication algebras. In the second part of the paper we extend our investigations to not necessarily lattice-ordered effect algebras satisfying the Ascending Chain Condition, in particular to finite effect algebras. Also in this case we characterize the operation of implication in a similar way as it was done in the lattice case. Finally, we provide an algebraic semantics of effect implication algebras.

For concepts and results concerning effect algebras the reader is referred to the monograph \cite{DP} by A.~Dvure\v censkij and S.~Pulmannov\'a. The following definition and lemma are taken from \cite{DP}

\begin{definition}
An {\em effect algebra} is a partial algebra $\mathbf E=(E,+,{}',0,1)$ of type $(2,1,0,0)$ where $(E,{}',0,1)$ is an algebra and $+$ is a partial operation satisfying the following conditions for all $x,y,z\in E$:
\begin{enumerate}[{\rm(E1)}]
\item if $x+y$ is defined then so is $y+x$ and $x+y=y+x$,
\item $(x+y)+z$ is defined if and only if so is $x+(y+z)$, and in this case $(x+y)+z=x+(y+z)$,
\item $x+y=1$ if and only if $y=x'$.
\item if $1+x$ is defined then $x=0$.
\end{enumerate}
On $E$ a binary relation $\leq$ can be defined by
\[
x\leq y\text{ if there exists some }z\in E\text{ with }x+z=y
\]
{\rm(}$x,y\in E${\rm)}. Then $(E,\leq,0,1)$ becomes a bounded poset and $\leq$ is called the {\em induced order} of $\mathbf E$. If $(E,\leq)$ is a lattice then $\mathbf E$ is called a lattice effect algebra.
\end{definition}

\begin{lemma}\label{lem1}
If $(E,+,{}',0,1)$ is an effect algebra, $\leq$ its induced order and $a,b,c\in E$ then
\begin{enumerate}[{\rm(i)}]
\item $a\leq b$ implies $b'\leq a'$,
\item $a''=a$,
\item $a+b$ is defined if and only if $a\leq b'$,
\item if $a\leq b$ and $b+c$ is defined then $a+c$ is defined and $a+c\leq b+c$,
\item if $a+c$ and $b+c$ are defined then $a+c\leq b+c$ if and only if $a\leq b$,
\item if $a\leq b$ then $a+(a+b')'=b$ and $(b'+(b'+a)')'=a$,
\item $a+0=a$,
\item $0'=1$ and $1'=0$.
\end{enumerate}
\end{lemma}

\section{Lattice effect implication algebras}

For lattice effect algebras $(E,+,{}',0,1)$, the {\em Sasaki implication} (called also {\em material implication} in [3]) was introduced as follows:
\[
x\rightarrow y:=x'+(x\wedge y)
\]
for all $x,y\in E$. For our sake, we will consider this in the following way:
\[
x\rightarrow y:=y+(x\vee y)'=y+(x'\wedge y')
\]
for all $x,y\in E$ and call this kind of implication {\em natural implication}. The reason for this choice is that it turns out that some computations become more feasible with this natural implication than with the Sasaki one.

In the rest of our paper, the implication investigated in lattice effect algebras will be the natural one.

\begin{theorem}\label{th3}
Let $(E,+,{}',0,1)$ be a lattice effect algebra and $a,b,c\in E$. Then
\begin{enumerate}[{\rm(i)}]
\item $a\leq b$ if and only if $a\rightarrow b=1$,
\item if $a\leq b'$ then $a+b=a'\rightarrow b$,
\item if $a\geq b$ then $a\rightarrow b=a'+b$,
\item if $a\geq b$ then $a\rightarrow b=b'\rightarrow a'$,
\item if $a\leq b$ then $b\rightarrow c\leq a\rightarrow c$.
\end{enumerate}
\end{theorem}

\begin{proof}
\
\begin{enumerate}[(i)]
\item The following are equivalent:
\begin{align*}
             a & \leq b, \\
       a\vee b & =b, \\
    (a\vee b)' & =b', \\
  b+(a\vee b)' & =1, \\
a\rightarrow b & =1.
\end{align*}
\item If $a\leq b'$ then $a'\rightarrow b=b+(a'\vee b)'=b+a''=b+a=a+b$.
\item If $a\geq b$ then $a\rightarrow b=a''\rightarrow b=a'+b$.
\item If $a\geq b$ then $a\rightarrow b=a'+b=b+a'=b''+a'=b'\rightarrow a'$.
\item If $a\leq b$ then $b\rightarrow c=c+(b\vee c)'\leq c+(a\vee c)'=a\rightarrow c$.
\end{enumerate}
\end{proof}

Concerning the reduct restricted to $\rightarrow$ and $\vee$, we can prove the following.

\begin{theorem}\label{th4}
Let $(E,+,{}',0,1)$ be a lattice effect algebra. Then the following identities hold in $(E,\vee,\rightarrow,{}',0,1)$:
\begin{enumerate}[{\rm(i)}]
\item $x\rightarrow0\approx x'$,
\item $1\rightarrow x\approx x$,
\item $x\rightarrow(y\rightarrow x)\approx1$,
\item $(x\rightarrow y)\rightarrow y\approx x\vee y$,
\item $((x\rightarrow y)\rightarrow y)\rightarrow y\approx x\rightarrow y$,
\item $x\rightarrow((x\rightarrow y)\rightarrow y)\approx1$,
\item $y\rightarrow((x\rightarrow y)\rightarrow y)\approx1$,
\item $y'\rightarrow((x\rightarrow y)\rightarrow y)'\approx x\rightarrow y$.
\end{enumerate}
\end{theorem}

\begin{proof}
We have
\begin{enumerate}[(i)]
\item $x\rightarrow0\approx0+(x\vee0)'\approx x'$,
\item $1\rightarrow x\approx x+(1\vee x)'\approx x$,
\item $x\leq x+(y\vee x)'=y\rightarrow x$,
\item $(x\rightarrow y)\rightarrow y\approx y+((y+(x\vee y)')\vee y)'\approx y+(y+(x\vee y)')'\approx x\vee y$ according to (vi) of Lemma~\ref{lem1},
\item $((x\rightarrow y)\rightarrow y)\rightarrow y\approx(x\rightarrow y)\vee y\approx x\rightarrow y$ according to (i) of Theorem~\ref{th3} and (iii),
\item $x\leq x\vee y=(x\rightarrow y)\rightarrow y$ according to (iv),
\item $y\leq x\vee y=(x\rightarrow y)\rightarrow y$ according to (iv),
\item $y'\rightarrow((x\rightarrow y)\rightarrow y)'\approx y'\rightarrow(x\vee y)'\approx y''+(x\vee y)'\approx y+(x\vee y)'\approx x\rightarrow y$ according to (iv), to (iii) of Theorem~\ref{th3} and to (ii) of Lemma~\ref{lem1}.
\end{enumerate}
\end{proof}

Summarizing the above properties, we can introduce an alter ego of a lattice effect algebra, i.e.\ its implication version.

\begin{definition}\label{def1}
A {\em lattice effect implication algebra} is an algebra $(I,\rightarrow,0)$ of type $(2,0)$ satisfying the following conditions for all $x,y,z\in I$ {\rm(}we abbreviate $x\rightarrow0$ by $x'$ and $0'$ by $1${\rm)}:
\begin{enumerate}[{\rm(i)}]
\item $0\rightarrow x\approx x\rightarrow x\approx x\rightarrow1\approx1$,
\item if $x\rightarrow y=y\rightarrow x=1$ then $x=y$,
\item if $x\rightarrow y=y\rightarrow z=1$ then $x\rightarrow z=1$,
\item if $x\rightarrow y=1$ then $y'\rightarrow x'=1$,
\item $x''\approx x$,
\item $x\rightarrow((x\rightarrow y)\rightarrow y)\approx1$,
\item $y\rightarrow((x\rightarrow y)\rightarrow y)\approx1$,
\item if $x\rightarrow z=y\rightarrow z=1$ then $((x\rightarrow y)\rightarrow y)\rightarrow z=1$,
\item if $x\rightarrow y=1$ then $y\rightarrow x=x'\rightarrow y'$,
\item $x\rightarrow y'=(x'\rightarrow y)\rightarrow z'=1$ if and only if $y\rightarrow z'=x\rightarrow(y'\rightarrow z)'=1$, and in this case $(x'\rightarrow y)'\rightarrow z=x'\rightarrow(y'\rightarrow z)$,
\item $y'\rightarrow((x\rightarrow y)\rightarrow y)'\approx x\rightarrow y$,
\item $x\rightarrow(y\rightarrow x)\approx1$.
\end{enumerate}
\end{definition}

In order to show the correspondence between a lattice effect algebra and the implication version, we state and prove the following three theorems.

\begin{theorem}\label{th1}
Let $\mathbf E=(E,+,{}',0,1)$ be a lattice effect algebra with induced order $\leq$ and lattice operations $\vee$ and $\wedge$ and put
\[
x\rightarrow y:=y+(x\vee y)'
\]
for all $x,y\in E$. Then $\mathbb{IL}(\mathbf E):=(E,\rightarrow,0)$ is a lattice effect implication algebra where
\begin{align*}
& x\leq y\text{ if and only if }x\rightarrow y=1, \\
& x\vee y\approx(x\rightarrow y)\rightarrow y\text{ and }x\wedge y\approx(x'\vee y')'.
\end{align*}
for all $x,y\in E$.
\end{theorem}

\begin{proof}
The proof follows from Theorems~\ref{th3} and \ref{th4}. More precisely, (i) -- (iii) follow from the fact that $(E,\leq,0,1)$ is a bounded poset, (iv) follows from the fact that $'$ is antitone on $(E,\leq)$, (v) from the fact that $'$ is an involution, (vi) and (vii) follow from the fact that $x\vee y$ is an upper bound of $x$ and $y$, (viii) follows from the fact that $x\vee y$ is less than or equal to every upper bound of $x$ and $y$, (ix) from (E1), (x) from (E2), (xi) from (viii) of Theorem~\ref{th4} and (xii) from (iii) of Theorem~\ref{th4}.
\end{proof}

\begin{theorem}\label{th2}
Let $\mathbf I=(I,\rightarrow,0)$ be a lattice effect implication algebra and put
\begin{align*}
 x' & :=x\rightarrow0, \\
  1 & :=0', \\
x+y & :=x'\rightarrow y\text{ if and only if }x\rightarrow y'=1
\end{align*}
for all $x,y\in E$. Then $\mathbb{EL}(\mathbf I):=(I,+,{}',0,1)$ is a lattice effect algebra with induced order $\leq$ and lattice operations $\vee$ and $\wedge$ where
\begin{align*}
& x\leq y\text{ if and only if }x\rightarrow y=1, \\
& x\vee y\approx(x\rightarrow y)\rightarrow y\text{ and }x\wedge y\approx(x'\vee y')'.
\end{align*}
\end{theorem}

\begin{proof}
Let $a,b,c\in I$. Define
\[
x\leq y\text{ if and only if }x\rightarrow y=1
\]
for all $x,y\in I$. Then, due to (i) -- (v) of Definition~\ref{def1}, $(I,\leq,{}',0,1)$ is a bounded poset with an antitone involution. If $a+b$ is defined then $a\leq b'$ and hence $b\leq a'$, i.e.\ $b+a$ is defined and
\[
a+b=a'\rightarrow b=b'\rightarrow a=b+a
\]
according to (ix) of Definition~\ref{def1}. This shows (E1). Now $a+b$ is defined if and only if $a\leq b'$, and in this case $a+b=a'\rightarrow b$. If $a+b$ is defined then $(a+b)+c$ is defined if and only if $a+b\leq c'$, and in this case
\[
(a+b)+c=(a+b)'\rightarrow c=(a'\rightarrow b)'\rightarrow c.
\]
Hence $(a+b)+c$ is defined if and only if both $a\leq b'$ and $a'\rightarrow b\leq c'$. On the other hand, $b+c$ is defined if and only if $b\leq c'$, and in this case $b+c=b'\rightarrow c$. If $b+c$ is defined then $a+(b+c)$ is defined if and only if $a\leq(b+c)'$, and in this case \[
a+(b+c)=a'\rightarrow(b+c)=a'\rightarrow(b'\rightarrow c).
\]
Hence $a+(b+c)$ is defined if and only if both $b\leq c'$ and $a\leq(b'\rightarrow c)'$. According to (x) of Definition~\ref{def1}, (E2) holds. Now the following are equivalent:
\begin{align*}
& a+b=1, \\
& a\leq b'\text{ and }a'\rightarrow b=1, \\
& b\leq a'\text{ and }a'\leq b, \\
& b=a'.
\end{align*}
This shows (E3). If $1+a$ is defined then $1\leq a'$ whence $a'=1$, i.e.\ $a=0$ showing (E4). Hence $\mathbb{EL}(\mathbf I)$ is an effect algebra. According to (vi) -- (viii) of Definition~\ref{def1} we have
\[
x\vee y\approx(x\rightarrow y)\rightarrow y,
\]
and because $'$ is an antitone involution on $(I,\leq)$ we have
\[
x\wedge y\approx(x'\vee y')'.
\]
If $a\leq b$ then $a\rightarrow(b\rightarrow a)''=a\rightarrow(b\rightarrow a)=1$ according to (xii) and
\begin{align*}
a+(b\rightarrow a)' & =a'\rightarrow(b\rightarrow a)'=(b\rightarrow a)''\rightarrow(a'\vee(b\rightarrow a)')'=(b\rightarrow a)\rightarrow a''= \\
& =(b\rightarrow a)\rightarrow a=b\vee a=b
\end{align*}
according to (xi). If, conversely, $a+c=b$ then $a\rightarrow c'=1$ and $a'\rightarrow c=b$ and hence
\[
a\leq c'\rightarrow a=c'\rightarrow a''=c'\rightarrow(a'\vee c)'=a'\rightarrow c=b
\]
according to (xii) and (xi). This shows that induced order of $\mathbb{EL}(\mathbf I)$ coincides with $\leq$.
\end{proof}

\begin{theorem}
The correspondence between lattice effect algebras and lattice effect implication algebras described by Theorems~\ref{th1} and \ref{th2} is one-to-one.
\end{theorem}

\begin{proof}
Let $\mathbf E=(E,+,{}',0,1)$ be a lattice effect algebra, put $\mathbb{IL}(\mathbf E)=(E,\rightarrow,0)$ and $\mathbb{EL}(\mathbb{IL}(\mathbf E))=(E,\oplus,{}^*,0,e)$ and let $a,b\in E$. Then $a\oplus b$ is defined if and only if so is $a+b$ since both are equivalent to $a\leq b'$, and in this case
\[
a\oplus b=a'\rightarrow b=a+b
\]
according to Theorem~\ref{th3}. Moreover,
\[
a^*=a\rightarrow0=a'
\]
according to Theorem~\ref{th4}, and
\[
e=0\rightarrow0=0'=1
\]
showing $\mathbb{EL}(\mathbb{IL}(\mathbf E))=\mathbf E$. Conversely, let $\mathbf I=(I,\rightarrow,0)$ be a lattice effect implication algebra, put $\mathbb{EL}(\mathbf I)=(I,+,{}',0,1)$ and $\mathbb{IL}(\mathbb{EL}(\mathbf I))=(I,\Rightarrow,0)$ and let $c,d\in I$. Then
\[
c\Rightarrow d=d+(c\vee d)'=d'\rightarrow((c\rightarrow d)\rightarrow d)'=c\rightarrow d
\]
according to (xi). This shows $\mathbb{IL}(\mathbb{EL}(\mathbf I))=\mathbf I$.
\end{proof}

\section{Axioms and rules for the logic of lattice effect algebras}

By a propositional language we understand some set $\mathcal L$ of propositional connectives. The arity of a propositional connective $c\in\mathcal L$ is called the rank of $c$. The $\mathcal L$-formulas are built in the usual way from the propositional variables using the connectives of $\mathcal L$. We denote the set of all $\mathcal L$-formulas by $\mathcal Fm_{\mathcal L}$ or $\mathcal Fm$ in brief when the language $\mathcal L$ is clear from the context.

By a logic in $\mathcal L$ we mean a standard deductive system over $\mathcal L$ or, equivalently, its consequence relation $\vdash_{\mathcal L}$ (see \cite{BP} for details). We shall use the notation $\vdash$ whenever there is no danger of confusion. 

Let us mention that a system of axioms and rules for the propositional logic induced by lattice effect algebras was already presented in \cite{RSS}. However, that system is rather complicated because it consists of five axioms and ten derivation rules. In what follows we present another system having only three axioms and five rules. Moreover, three of these rules, namely Modus Ponens, Suffixing and Weak Prefixing are common in many propositional calculi.

By the propositional logic $L_{{\rm LEA}}$ in a language $\mathcal L=\{\rightarrow,0\}$ ($\rightarrow$ is of rank 2, $0$ is of rank $0$) we understand a consequence relation $\vdash_{{\rm LEA}}$ (or $\vdash$, in brief) satisfying the axioms
\begin{enumerate}[(A1)]
\item $\vdash\varphi\rightarrow(\psi\rightarrow\varphi)$,
\item $\vdash((\varphi\rightarrow\psi)\rightarrow\psi)\rightarrow((\psi\rightarrow\varphi)\rightarrow\varphi)$,
\item $\vdash0\rightarrow\varphi$
\end{enumerate}
and the rules
\begin{enumerate}
\item[(MP)] $\varphi,\varphi\rightarrow\psi\vdash\psi$,
\item[(Sf)] $\varphi\rightarrow\psi\vdash(\psi\rightarrow\chi)\rightarrow(\varphi\rightarrow\chi)$,
\item[(WPf)] $\varphi\rightarrow\psi,\psi\rightarrow\varphi \vdash(\chi\rightarrow\varphi)\rightarrow(\chi\rightarrow\psi)$,
\item[(R1)] $\varphi\rightarrow\psi\vdash(\neg\varphi\rightarrow\neg\psi)\rightarrow(\psi\rightarrow\varphi)$,
\item[(R2)] $\varphi\rightarrow\neg\psi,(\neg\varphi\rightarrow\psi)\rightarrow\neg\chi\vdash(\neg(\neg\varphi\rightarrow\psi)\rightarrow \chi)\rightarrow (\neg\varphi\rightarrow(\neg\psi\rightarrow\chi))$,
\end{enumerate}
where $\neg\varphi:=\varphi\rightarrow0$ and $1:=\neg0=0\rightarrow0$.

The axiom system (A1) -- (A3) with the derivation rules (MP) -- (R2) will be referred to as {\em axiom system $\mathbf A$}.

As usual, for $\Gamma\cup\{\varphi,\psi\}\subseteq\mathcal Fm$, the biconditional $\Gamma\vdash\varphi\leftrightarrow\psi$ is an abbreviation for $\Gamma\vdash \varphi\rightarrow \psi$ and $\Gamma\vdash\psi\rightarrow \varphi$. 

In order to show that the system $\mathbf A$ is really an axiom system in Gentzen style for lattice effect algebras, we prove the following important properties.

\begin{theorem}\label{th9}
In the propositional logic $L_{{\rm LEA}}$ the following are provable:
\begin{enumerate}[{\rm(a)}]
\item $\varphi\rightarrow\psi,\psi\rightarrow\chi\vdash\varphi\rightarrow\chi$,
\item $\vdash\varphi\rightarrow1$,
\item $\vdash\varphi\rightarrow\varphi$,
\item $\vdash\psi\rightarrow(\psi\vee\varphi)$ where $\psi\vee\varphi:=(\psi\rightarrow\varphi)\rightarrow\varphi$,
\item $\vdash\neg\neg\varphi\leftrightarrow\varphi$,
\item $\varphi\rightarrow\psi\vdash\neg\varphi\rightarrow\neg\psi\leftrightarrow\psi\rightarrow\varphi$,
\item $\vdash\psi\rightarrow(\varphi\vee\psi)$ and $\vdash(\varphi\vee\psi)\leftrightarrow(\psi\vee\varphi)$,
\item $\varphi\rightarrow\psi,\chi\rightarrow\psi\vdash(\varphi\vee\chi)\rightarrow\psi$,
\item $\vdash\varphi\rightarrow\psi\leftrightarrow(\varphi\vee\psi)\rightarrow\psi$,
\item $\vdash\neg\varphi\rightarrow\neg(\psi\vee\varphi)\leftrightarrow\psi\rightarrow\varphi$.
\end{enumerate}
\end{theorem}

\begin{proof}
\
\begin{enumerate}[(a)]
\item This follows from (Sf) and (MP).
\item According to (A1),
\[
\vdash1\rightarrow(\varphi\rightarrow1)
\]
and according to (A3)
\[
\vdash0\rightarrow0,
\]
i.e.
\[
\vdash1.
\]
Thus applying (MP) we conclude
\[
\vdash\varphi\rightarrow1.
\]
\item According to (A1),
\begin{align*}
& \vdash\psi\rightarrow((\varphi\rightarrow\psi)\rightarrow\psi), \\
& \vdash\varphi\rightarrow(\psi\rightarrow\varphi),
\end{align*}
thus by (Sf),
\[
\vdash((\psi\rightarrow\varphi)\rightarrow\varphi)\rightarrow(\varphi\rightarrow\varphi).
\]
Applying (A2) and (Sf) we have
\[
\vdash(((\psi\rightarrow\varphi)\rightarrow\varphi)\rightarrow(\varphi\rightarrow\varphi))\rightarrow(((\varphi\rightarrow\psi)\rightarrow\psi)\rightarrow(\varphi\rightarrow\varphi)).
\]
Now (MP) yields
\[
\vdash((\varphi\rightarrow\psi)\rightarrow\psi)\rightarrow(\varphi\rightarrow\varphi).
\]
Hence
\[
\vdash\psi\rightarrow(\varphi\rightarrow\varphi)
\]
by (a). Substituting any provable formula for $\psi$ within the last formula yields
\[
\vdash\varphi\rightarrow\varphi
\]
according to (MP).
\item We have
\[
\vdash\psi\rightarrow(\varphi\vee\psi)
\]
according to (A1) and
\[
\vdash(\varphi\vee\psi)\rightarrow(\psi\vee\varphi)
\]
according to (A2), and hence
\[
\vdash\psi\rightarrow(\psi\vee\varphi)
\]
by (a).
\item We have
\[
\vdash\neg\neg\varphi\rightarrow(0\vee\varphi)
\]
according to (A2) and
\[
\vdash(0\vee\varphi)\rightarrow\varphi
\]
according to (A3), (d) and (MP). This shows
\[
\vdash\neg\neg\varphi\rightarrow\varphi
\]
by (a). Conversely,
\[
\vdash\varphi\rightarrow\neg\neg\varphi
\]
according to (d).
\item According to (R1) we have
\[
\varphi\rightarrow\psi\vdash(\neg\varphi\rightarrow\neg\psi)\rightarrow(\psi\rightarrow\varphi).
\]
Moreover,
\[
\vdash(\psi\rightarrow\varphi)\rightarrow(\neg\varphi\rightarrow\neg\psi)
\]
according to (Sf).
\item We have
\[
\vdash\psi\rightarrow(\varphi\vee\psi)
\]
according to (A1). The rest follows from (A2).
\item Applying (Sf) twice we obtain
\begin{align*}
                                  \varphi\rightarrow\psi & \vdash(\psi\rightarrow\chi)\rightarrow(\varphi\rightarrow\chi), \\
(\psi\rightarrow\chi)\rightarrow(\varphi\rightarrow\chi) & \vdash(\varphi\vee\chi)\rightarrow(\psi\vee\chi)
\end{align*}
and hence
\[
\varphi\rightarrow\psi\vdash(\varphi\vee\chi)\rightarrow(\psi\vee\chi)
\]
by (a). But
\[
\vdash(\psi\vee\chi)\rightarrow(\chi\vee\psi)
\]
according to (A2). Now
\[
\chi\rightarrow\psi\vdash(\chi\vee\psi)\rightarrow\psi
\]
according to (d). Hence (h) follows from (a).
\item We have
\[
\vdash(\varphi\rightarrow\psi)\rightarrow((\varphi\vee\psi)\rightarrow\psi)
\]
according to (d). Conversely,
\[
\vdash((\varphi\vee\psi)\rightarrow\psi)\rightarrow(\psi\vee(\varphi\rightarrow\psi))
\]
according to (A2). Because of (A1) we have
\[
\vdash\psi\rightarrow(\varphi\rightarrow\psi)
\]
and because of (d),
\[
\vdash(\psi\rightarrow(\varphi\rightarrow\psi))\rightarrow((\psi\vee(\varphi\rightarrow\psi))\rightarrow(\varphi\rightarrow\psi)).
\]
Now (MP) implies
\[
\vdash(\psi\vee(\varphi\rightarrow\psi))\rightarrow(\varphi\rightarrow\psi).
\]
By (a) we obtain
\[
\vdash((\varphi\vee\psi)\rightarrow\psi)\rightarrow(\varphi\rightarrow\psi).
\]
\item According to (i) we have
\[
(\psi\vee\varphi)\rightarrow\varphi\leftrightarrow\psi\rightarrow\varphi
\]
and according to (f),
\[
\varphi\rightarrow(\psi\vee\varphi)\vdash\neg\varphi\rightarrow\neg(\psi\vee\varphi)\leftrightarrow(\psi\vee\varphi)\rightarrow\varphi.
\]
But
\[
\vdash\varphi\rightarrow(\psi\vee\varphi)
\]
holds because of (A1). Now (j) follows from (a).
\end{enumerate}
\end{proof}

\section{Algebraic semantics}

For a class $\mathcal K$ of $\mathcal L$-algebras over a language $\mathcal L$, consider the relation $\models_\mathcal K$ that holds between a set $\Sigma$ of identities and a single identity $\varphi\approx\psi$ if every interpretation of $\varphi\approx\psi$ in a member of $\mathcal K$ holds provided each identity in $\Sigma$ holds under the same interpretation. In this case we say that $\varphi\approx\psi$ is a {\em $\mathcal K$-consequence} of $\Sigma$. The relation $\models_\mathcal K$ is called the {\em semantic equational consequence relation} determined by $\mathcal K$.

Given a deductive system $(\mathcal L,\vdash_L)$ over a language $\mathcal L$, a class $\mathcal K$ of $\mathcal L$-algebras is called an {\em algebraic semantics} for $(\mathcal L,\vdash_L)$ if $\vdash_L$ can be interpreted in $\models_\mathcal K$ in the following sense: There exists a finite system $\delta_i(p)\approx\varepsilon_i(p)$, ($\delta(\varphi)\approx\varepsilon(\varphi)$, in brief) of identities with a single variable $p$ such that for all $\Gamma\cup\{\varphi\}\subseteq\mathcal Fm$,
\[
\Gamma\vdash_L\varphi\Leftrightarrow\{\delta(\varphi)\approx\varepsilon(\varphi),\psi\in\Gamma\}\models_\mathcal K\delta(\varphi)\approx\varepsilon(\varphi).
\]
Then $\delta_i\approx\varepsilon_i$ are called {\em defining identities} for $(\mathcal L,\vdash_L)$ and $\mathcal K$.

$\mathcal K$ is said to be {\em equivalent} to $(\mathcal L,\vdash_L)$ if there exists a finite system $\Delta_j(p,q)$ of formulas with two variables $p,q$ such that for every identity $\varphi\approx\psi$,
\[
\varphi\approx\psi\tautequiv_\mathcal K\delta(\varphi\Delta\psi)\approx\varepsilon(\varphi\Delta\psi),
\]
where $\varphi\Delta\psi$ means just $\Delta(\varphi,\psi)$ and $\Gamma\tautequiv_\mathcal K\Delta$ is an abbreviation for the conjunction $\Gamma\models_\mathcal K\Delta$ and $\Delta\models_\mathcal K\Gamma$. 

According to \cite R and \cite{RS}, a {\em standard system of implicative extensional propositional calculus} (SIC, for short) is a deductive system $(\mathcal L,\vdash_L)$ satisfying the following conditions:
\begin{itemize}
\item The language $\mathcal L$ contains a finite number of connectives of rank $0$, $1$ and $  2$ and none of higher rank,
\item $\mathcal L$ contains a binary connective $\rightarrow$ for which the following theorems and derived inference rules hold:
\begin{align*}
& \vdash\varphi\rightarrow\varphi, \\
& \varphi,\varphi\rightarrow\psi\vdash\psi, \\
& \varphi\rightarrow\psi,\psi\rightarrow\chi\vdash\varphi\rightarrow\chi, \\
& \varphi\rightarrow\psi,\psi\rightarrow\varphi\vdash P(\varphi)\rightarrow P(\psi)\text{ for every unary }P\in\mathcal L, \\
& \varphi\rightarrow\psi,\psi\rightarrow\varphi,\chi\rightarrow\lambda,\lambda\rightarrow\chi\vdash Q(\varphi,\chi)\rightarrow Q(\psi,\lambda)\text{ for every binary }Q\in\mathcal L.
\end{align*}
\end{itemize}
As one can immediately see, the system $(\mathcal L,\vdash_{{\rm LEA}})$ fulfils all the properties of SIC. Indeed, there is no unary connective in our language, and we have the only binary connective $\rightarrow$. Now the first property is (c) of Theorem 3.1, the second is (MP) and the third is (a) from Theorem~\ref{th9}. To show the fifth property, we have $\varphi\rightarrow\psi\vdash (\varphi\rightarrow\chi)\rightarrow (\psi\rightarrow \chi)$ by (Sf), and $\chi\rightarrow\lambda,\lambda\rightarrow \chi\vdash (\psi\rightarrow\chi)\rightarrow (\psi\rightarrow \lambda)$ by (WPf). The rest then follows again by (a) of Theorem~\ref{th9}.

Moreover, taking $\varepsilon(p)=p$, $\delta(p)=p\rightarrow p$ and $\Delta(p,q)=\{p\rightarrow q,q\rightarrow p\}$, it is known (see \cite{BP}) that every SIC has an equivalent algebraic semantics with the defining identity $\delta\approx\varepsilon$ and with the set $\Delta$ as an equivalence system. As a consequence we obtain

\begin{proposition}\label{prop1}
The logic $(\mathcal L,\vdash_{{\rm LEA}})$ is algebraizable with equivalence formulas $\Delta=\{p\rightarrow q,q\rightarrow p\}$ and the defining identity $p\approx p\rightarrow p$.
\end{proposition}

In order to show that LEA is an equivalent algebraic semantics for $(\mathcal L,\vdash_{{\rm LEA}})$ we use the following statement (see \cite{BP}, Theorem~2.17).

\begin{proposition}\label{prop2}
Let $(\mathcal L,\vdash_L)$ be a deductive system given by a set of axioms {\rm Ax} and a set of inference rules {\rm Ir}. Assume $(\mathcal L,\vdash_L)$ is algebraizable with equivalence formulas $\Delta$ and defining identities $\delta\approx\varepsilon$. Then the unique equivalent semantics for $(\mathcal L,\vdash_L)$ is axiomatized by the identities
\begin{itemize}
\item $\delta(\varphi)\approx\varepsilon(\varphi)$ for each $\varphi\in{\rm Ax}$,
\item $\delta(p\Delta p)\approx\varepsilon(p\Delta p)$
\end{itemize}
together with the quasiidentities
\begin{itemize}
\item $\delta(\psi_0)\approx\varepsilon(\psi_0)\wedge\cdots\wedge\delta(\psi_{n-1})\approx\varepsilon(\psi_{n-1})\Rightarrow\delta(\varphi)\approx\varepsilon(\varphi)$ for each $\psi_0,\ldots,\psi_{n-1}\vdash\varphi\in{\rm Ir}$,
\item $\delta(p\Delta q)\approx\varepsilon(p\Delta q)\Rightarrow p\approx q$.
\end{itemize}
\end{proposition}

Taking into account that $\vdash\varphi\rightarrow\varphi$ in $(\mathcal L,\vdash_{{\rm LEA}})$, we have $p\rightarrow p\approx1$, i.e.\ $\varepsilon (p)=1$. Applying the previous proposition, we are going to prove the following.

\begin{theorem}\label{th10}
The equivalent algebraic semantics for $(\mathcal L,\vdash_{{\rm LEA}})$ is axiomatized by the following identities and quasiidentities:
\begin{enumerate}[{\rm(1)}]
\item $x\rightarrow(y\rightarrow x)\approx1$,
\item $((x\rightarrow y)\rightarrow y)\rightarrow((y\rightarrow x)\rightarrow x)\approx1$,
\item $0\rightarrow x\approx1$,
\item $x\rightarrow x\approx1$,
\item $x=x\rightarrow y=1\Rightarrow y=1$,
\item $x\rightarrow y=1\Rightarrow(y\rightarrow z)\rightarrow(x\rightarrow z)=1$,
\item $x\rightarrow y=1\Rightarrow(x'\rightarrow y')\rightarrow(y\rightarrow x)=1$,
\item $x\rightarrow y=y\rightarrow x=1\Rightarrow x=y$,
\item $x\rightarrow y'=(x'\rightarrow y)\rightarrow z'=1\Rightarrow((x'\rightarrow y)'\rightarrow z)\rightarrow (x'\rightarrow (y'\rightarrow z))=1$.
\end{enumerate}
\noindent
This system just corresponds to the quasivariety of lattice effect implication algebras. 
\end{theorem}

\begin{proof}
It can be immediately seen that any lattice effect implication algebra fulfills the above axioms. 
We will prove the converse, i.e. that the properties listed above yield the conditions of Definition~\ref{def1}.
\begin{enumerate}[(i)]
\item[(i)] The first two identities of (i) are just (3) and (4) and the remaining one follows from (4) and (1).
\item[(ii)] This is just (8).
\item[(iii)] Assume $x\rightarrow y=y\rightarrow z=1$. Then according to (6) we have $(y\rightarrow z)\rightarrow(x\rightarrow z)=1$ which according to (5) yields $x\rightarrow z=1$.

(i) -- (iii) show that the relation $\leq$ defined by $x\leq y$ if and only if $x\rightarrow y=1$ is an partial order relation with smallest element $0$ and greatest element $1$.
\item[(iv)] This follows from (6).
\item[(vi)] According to (1) and (2) we have $x\leq(y\rightarrow x)\rightarrow x\leq(x\rightarrow y)\rightarrow y$.
\item[(vii)] According to (1) we have $y\leq(x\rightarrow y)\rightarrow y$.
\item[(v)] We have
\[
(1\rightarrow x)\rightarrow x\approx((0\rightarrow x)\rightarrow x)\rightarrow x\approx(x\rightarrow(0\rightarrow x))\rightarrow(0\rightarrow x)\approx1\rightarrow1\approx1
\]
according to (3), (2), (1) and (4) and therefore $x\leq(0\rightarrow x)\rightarrow x=1\rightarrow x\leq x$ according to (1) and (3), i.e.\ $1\rightarrow x\approx x$. Now
\[
x''\approx(x\rightarrow0)\rightarrow0\approx(0\rightarrow x)\rightarrow x\approx1\rightarrow x\approx x
\]
according to (2) and (3).
\item[(viii)]
Assume $x\leq z$ and $y\leq z$. Then according to (6) and (2) we have $z\rightarrow y\leq x\rightarrow y$ and
\[
(x\rightarrow y)\rightarrow y\leq(z\rightarrow y)\rightarrow y=(y\rightarrow z)\rightarrow z=1\rightarrow z=z.
\]
\item[(ix)] Assume $x\leq y$. Then $y'\leq x'$ according to (6) and
\[
y\rightarrow x=y''\rightarrow x''\leq x'\rightarrow y'\leq y\rightarrow x
\]
according to (v) and (7).
\item[(x)] Assume  $x\rightarrow y'=(x'\rightarrow y)\rightarrow z'=1$, i.e. $x\leq y'$ and $x'\rightarrow y\leq z'$. We know by (1) that $y\leq x'\rightarrow y$, thus $y\leq z'$. Now, we can apply (7) and (v) to obtain $y'\rightarrow z=z'\rightarrow y$. Using (6), $x'\rightarrow y\leq z'$ yields $z'\rightarrow y\leq (x'\rightarrow y)\rightarrow y=x'\vee y=x'$. This shows that $y'\rightarrow z\leq x'$ which due to (iv) gives $x\leq (y'\rightarrow z)'$. The converse implication follows by interchanging the elements $x,y,z$. Finally, under the above assumptions we have by (9) $(x'\rightarrow y)'\rightarrow z\leq x'\rightarrow (y'\rightarrow z)$. Again, by interchanging the elements $x,y,z$ we obtain the converse inequality and thus equality.
\item[(xi)] Applying (vii), (ix), (2) and (1) we obtain
\begin{align*}
y'\rightarrow((x\rightarrow y)\rightarrow y)' & \approx((x\rightarrow y)\rightarrow y)\rightarrow y\approx(y\rightarrow(x\rightarrow y))\rightarrow(x\rightarrow y)\approx \\
                                              & \approx1\rightarrow(x\rightarrow y)\approx x\rightarrow y.
\end{align*}
\item[(xii)]  This is just (1).
\end{enumerate}
\end{proof}

We conclude that, using the equivalence between lattice effect algebras and lattice effect implication algebras and Theorem~\ref{th10}, system $\mathbf A$ is an algebraic axiomatization of the logic of lattice effect algebras.

\section{Effect implication algebras}

Although lattice effect algebras are more feasible for some kinds of algebraic investigation, effect algebras which need not be lattice-ordered play a more important role in algebraic axiomatization of the logic of quantum mechanics. Hence, it is a question if also in this case we are able to derive some semantics, axioms and rules for a Gentzen type axiomatization in a way similar to that for lattice-ordered effect algebras. 

The first question is how to define the connective implication in not lattice-ordered effect algebras. In this section we show that this can be done successfully in the case when the effect algebra $(E,+,{}',0)$ in question is finite, or, more generally, if for every $x,y\in E$ there exist maximal elements in the lower cone $L(x',y')$.

In the following let $\mathbf E=(E,+,{}',0,1)$ be an effect algebra with induced order $\leq$ such that $(E,\leq)$ satisfies the Ascending Chain Condition (shortly, ACC). This condition says that in  $(E,\leq)$ there do not exist infinite ascending chains. Hence, if $(E,\leq)$ satisfies the ACC then every non-empty subset of $E$ has at least one maximal element. In particular, this is true in case of finite $E$.

Now let $a,b\in E$ and $A,B\subseteq E$. Here and in the following
\begin{align*}
& \Max A\text{ denotes the set of all maximal elements of }(A,\leq), \\
& A+B:=\{x+y\mid x\in A,y\in B\}, \\
& A\leq B\text{ means }x\leq y\text{ for all }x\in A\text{ and }y\in B, \\
& \text{we often identify }\{a\}\text{ with }a.
\end{align*}
Moreover, we define
\begin{align*}
a\rightarrow b & :=b+\Max L(a',b'), \\
a\rightarrow A & :=\bigcup_{x\in A}(a\rightarrow x).
\end{align*}

Since $(E,\leq)$ satisfies the ACC we have $x\rightarrow y\neq\emptyset$ for all $x,y\in E$. Moreover, if $\mathbf E$ is a lattice effect algebra then the definition of $\rightarrow$ coincides with the original natural implication $\rightarrow$ since then
\[
x\rightarrow y=y+\Max L(x',y')=y+\Max L(x'\wedge y')=y+(x'\wedge y')=y+(x\vee y)'
\]
for all $x,y\in E$.

Since every element of $\Max L(x',y')$ is less than or equal to $y'$, $y+\Max L(x',y')$ is defined for each $x,y\in E$. It is worth noticing that now $x\rightarrow y$ need not be a single element of $E$, but a non-void subset of $E$. Hence, one cannot expect that such an implication will satisfy the same properties as the implication in lattice effect algebras. On the other hand, it was already successfully used by the first and third author in \cite{CLb} in order to show that for the implication defined in this way the effect algebra in question can be organized in an operator residuated structure.

Although not all the conditions of Theorems~\ref{th3} and \ref{th4} are valid for effect algebra which need not be lattice-ordered, we still can prove several characteristic properties.

\begin{theorem}\label{th5}
Let $(E,+,{}',0,1)$ be an effect algebra satisfying the {\rm ACC} and $a,b\in E$. Then
\begin{enumerate}[{\rm(i)}]
\item $a\leq b$ if and only if $a\rightarrow b=1$,
\item if $a\leq b'$ then $a+b=a'\rightarrow b$,
\item if $a\geq b$ then $a\rightarrow b=a'+b$,
\item if $a\geq b$ then $a\rightarrow b=b'\rightarrow a'$,
\item $a\rightarrow0=a'$,
\item $1\rightarrow a=a$,
\item $b'\rightarrow\Max L(a',b')=a\rightarrow b$.
\end{enumerate}
\end{theorem}

\begin{proof}
\
\begin{enumerate}[(i)]
\item The following are equivalent:
\begin{align*}
              a & \leq b, \\
             b' & \leq a', \\
  \Max L(a',b') & =b', \\
b+\Max L(a',b') & =1, \\
 a\rightarrow b & =1.
\end{align*}
\item If $a\leq b'$ then $a'\rightarrow b=b+\Max L(a,b')=b+\Max L(a)=b+a=a+b$.
\item If $a\geq b$ then $a\rightarrow b=a''\rightarrow b=a'+b$.
\item If $a\geq b$ then $a\rightarrow b=a'+b=b+a'=b''+a'=b'\rightarrow a'$.
\item We have $a\rightarrow0=0+\Max L(a',0')=0+\Max L(a')=0+a'=a'$.
\item We have $1\rightarrow a=a+\Max L(1',a')=a+\Max L(0)=a+0=a$.
\item We have
\begin{align*}
b'\rightarrow\Max L(a',b') & =\{b'\rightarrow x\mid x\in\Max L(a',b')\}=\{b''+x\mid x\in\Max L(a',b')\}= \\
& =\{b+x\mid x\in\Max L(a',b')\}=b+\Max L(a',b')=a\rightarrow b.
\end{align*}
\end{enumerate}
\end{proof}

Summarizing all what was stated for implication in effect algebras, we can introduce the following concept.

\begin{definition}\label{def2}
An {\em effect implication algebra} is an algebra $(I,\rightarrow,0)$ of type $(2,0)$ satisfying the following conditions for all $x,y,z\in I$ {\rm(}we abbreviate $x\rightarrow0$ by $x'$ and $0'$ by $1${\rm)}:
\begin{enumerate}[{\rm(i)}]
\item $0\rightarrow x\approx x\rightarrow x\approx x\rightarrow1\approx1$,
\item if $x\rightarrow y=y\rightarrow x=1$ then $x=y$,
\item if $x\rightarrow y=y\rightarrow z=1$ then $x\rightarrow z=1$,
\item if $x\rightarrow y=1$ then $y'\rightarrow x'=1$,
\item $x''\approx x$,
\item if $x\rightarrow y=1$ then $y\rightarrow x=x'\rightarrow y'$,
\item $x\rightarrow y'=(x'\rightarrow y)\rightarrow z'=1$ if and only if $y\rightarrow z'=x\rightarrow(y'\rightarrow z)'=1$, and in this case $(x'\rightarrow y)'\rightarrow z=x'\rightarrow(y'\rightarrow z)$,
\item $x\rightarrow(y\rightarrow x)\approx1$
\end{enumerate}
and satisfying the condition that there does not exist an infinite sequence $a_1,a_2,a_3,\ldots$ of pairwise distinct elements of $I$ satisfying $a_n\rightarrow a_{n+1}=1$ for all positive integers $n$.
\end{definition}

Of course, every finite effect implication algebra satisfies trivially the last condition of Definition~\ref{def2}.

We are going to show that every effect algebra satisfying the ACC induces an effect implication algebra and vice versa.

\begin{theorem}\label{th6}
Let $\mathbf E=(E,+,{}',0,1)$ be an effect algebra satisfying the {\rm ACC} with induced order $\leq$ and put
\[
x\rightarrow y:=y+\Max L(x',y')
\]
for all $x,y\in E$. Then $\mathbb I(\mathbf E):=(E,\rightarrow,0)$ is an effect implication algebra. Moreover,
\[
x\leq y\text{ if and only if }x\rightarrow y=1.
\]
\end{theorem}

\begin{proof}
The proof follows from Theorem~\ref{th5}. More precisely, (i) -- (iii) follow from the fact that $(E,\leq,0,1)$ is a bounded poset, (iv) follows from the fact that $'$ is antitone on $(E,\leq)$, (v) follows from the fact that $'$ is an involution, (vi) follows from (E1), (vii) from (E2), (vii) from (vii) of Theorem~\ref{th5} and (viii) from the definition of $\rightarrow$.
\end{proof}

Although $x\rightarrow y$ need not be a single element of the corresponding effect algebra $\mathbf E$, we are able to prove that an effect implication algebra can be converted into an effect algebra where the partial operation $+$ is defined in the standard way.

\begin{theorem}\label{th7}
Let $\mathbf I=(I,\rightarrow,0)$ be an effect implication algebra and put
\begin{align*}
 x' & :=x\rightarrow0, \\
  1 & :=0', \\
x+y & :=x'\rightarrow y\text{ if and only if }x\rightarrow y'=1
\end{align*}
for all $x,y\in E$. Then $\mathbb E(\mathbf I):=(I,+,{}',0,1)$ is an effect algebra satisfying the {\rm ACC} with induced order $\leq$ where
\[
x\leq y\text{ if and only if }x\rightarrow y=1.
\]
\end{theorem}

\begin{proof}
Let $a,b,c\in I$. Define
\[
x\leq y\text{ if and only if }x\rightarrow y=1
\]
for all $x,y\in I$. Then, due to (i) -- (v) of Definition~\ref{def2}, $(I,\leq,{}',0,1)$ is a bounded poset with an antitone involution. If $a+b$ is defined then $a\leq b'$ and hence $b\leq a'$, i.e.\ $b+a$ is defined and
\[
a+b=a'\rightarrow b=b'\rightarrow a=b+a
\]
according to (vi) of Definition~\ref{def2}. This shows (E1). Now $a+b$ is defined if and only if $a\leq b'$, and in this case $a+b=a'\rightarrow b$. If $a+b$ is defined then $(a+b)+c$ is defined if and only if $a+b\leq c'$, and in this case
\[
(a+b)+c=(a+b)'\rightarrow c=(a'\rightarrow b)'\rightarrow c.
\]
Hence $(a+b)+c$ is defined if and only if both $a\leq b'$ and $a'\rightarrow b\leq c'$. On the other hand, $b+c$ is defined if and only if $b\leq c'$, and in this case $b+c=b'\rightarrow c$. If $b+c$ is defined then $a+(b+c)$ is defined if and only if $a\leq(b+c)'$, and in this case \[
a+(b+c)=a'\rightarrow(b+c)=a'\rightarrow(b'\rightarrow c).
\]
Hence $a+(b+c)$ is defined if and only if both $b\leq c'$ and $a\leq(b'\rightarrow c)'$. According to (vii) of Definition~\ref{def2}, (E2) holds. Now the following are equivalent:
\begin{align*}
& a+b=1, \\
& a\leq b'\text{ and }a'\rightarrow b=1, \\
& b\leq a'\text{ and }a'\leq b, \\
& b=a'.
\end{align*}
This shows (E3). If $1+a$ is defined then $1\leq a'$ whence $a'=1$, i.e.\ $a=0$ showing (E4). Hence $\mathbb E(\mathbf I)$ is an effect algebra satisfying the ACC whose induced order coincides with $\leq$ because of (ix) of Definition~\ref{def2}.
\end{proof}

The following theorem shows that the correspondence between an effect algebra satisfying the ACC and its corresponding effect implication algebra is almost one-to-one.

\begin{theorem}\label{th8}
The following hold:
\begin{enumerate}[{\rm(i)}]
\item If $\mathbf E$ is an effect algebra satisfying the {\rm ACC} then $\mathbb E(\mathbb I(\mathbf E))=\mathbf E$,
\item if $\mathbf I$ is an effect implication algebra satisfying the identity
\begin{enumerate}[{\rm(11)}]
\item $x\rightarrow y\approx y'\rightarrow\Max L(x',y')$
\end{enumerate}
then $\mathbb I(\mathbb E(\mathbf I))=\mathbf I$.
\end{enumerate}
\end{theorem}

\begin{proof}
Let $\mathbf E=(E,+,{}',0,1)$ be an effect algebra satisfying the ACC, put $\mathbb I(\mathbf E)=(E,\rightarrow,0)$ and $\mathbb E(\mathbb I(\mathbf E))=(E,\oplus,{}^*,0,e)$ and let $a,b\in E$. Then $a\oplus b$ is defined if and only if so is $a+b$ since both are equivalent to $a\leq b'$, and in this case
\begin{align*}
a\oplus b & =a'\rightarrow b=a+b, \\
      a^* & =a\rightarrow0=a', \\
        e & =0\rightarrow0=0'=1
\end{align*}
according to Theorem~\ref{th5} showing $\mathbb E(\mathbb I(\mathbf E))=\mathbf E$. Conversely, let $\mathbf I=(I,\rightarrow,0)$ be an effect implication algebra satisfying identity (11), put $\mathbb E(\mathbf I)=(I,+,{}',0,1)$ and $\mathbb I(\mathbb E(\mathbf I))=(I,\Rightarrow,0)$ and let $c,d\in I$. Then
\begin{align*}
c\Rightarrow d & =d+\Max L(c',d')=\{d+x\mid x\in\Max(c',d')\}=\{d'\rightarrow x\mid x\in\Max(c',d')\}= \\
& =d'\rightarrow\Max L(c',d')=c\rightarrow d
\end{align*}
according to (10). This shows $\mathbb I(\mathbb E(\mathbf I))=\mathbf I$.
\end{proof}

Theorem~\ref{th8} (i) says that every effect algebra satisfying the ACC is fully determined by its derived effect implication algebra. This enables us to set up the semantics and axiomatization for finite effect algebras which need not be lattice-ordered in a way similar to that for lattice effect algebras.

\section{The logic of finite effect algebras}

In this section we use the same general theory taken from \cite{BP} as in Sections~3 and 4. Hence, we need not repeat all what was said in the beginning of these two sections. All what was described in Propositions~\ref{prop1} and \ref{prop2} remains valid also here with the same equivalence system.

By the propositional logic $L_{{\rm EA}}$ in a language $\mathcal L=\{\rightarrow,0\}$ ($\rightarrow$ is of rank 2, $0$ is of rank $0$) we understand a consequence relation $\vdash_{{\rm EA}}$ (or $\vdash$, in brief) satisfying the axioms
\begin{enumerate}[(B1)]
\item $\vdash\varphi\rightarrow(\psi\rightarrow\varphi)$,
\item $\vdash\varphi\rightarrow\varphi$,
\item $\vdash\varphi\leftrightarrow(\varphi\rightarrow0)\rightarrow0$,
\item $\vdash0\rightarrow\varphi$
\end{enumerate}
and the rules
\begin{enumerate}
\item[(MP)] $\varphi,\varphi\rightarrow\psi\vdash\psi$,
\item[(Sf)] $\varphi\rightarrow\psi\vdash(\psi\rightarrow\chi)\rightarrow(\varphi\rightarrow\chi)$,
\item[(WPf)] $\varphi\rightarrow\psi,\psi\rightarrow\varphi \vdash(\chi\rightarrow\varphi)\rightarrow(\chi\rightarrow\psi)$,
\item[(R1)] $\varphi\rightarrow\psi\vdash(\neg\varphi\rightarrow\neg\psi)\rightarrow(\psi\rightarrow\varphi)$,
\item[(R2)] $\varphi\rightarrow\neg\psi,(\neg\varphi\rightarrow\psi)\rightarrow\neg\chi\vdash(\neg(\neg\varphi\rightarrow\psi)\rightarrow \chi)\rightarrow (\neg\varphi\rightarrow(\neg\psi\rightarrow\chi))$,
\end{enumerate}
where $\neg\varphi:=\varphi\rightarrow0$ and $1:=\neg0=0\rightarrow0$.

This axiom system will be referred to as system $\mathbf B$. We can prove the following consequences of $\mathbf B$.
 
\begin{theorem}
In the propositional logic $L_{{\rm EA}}$ the following are provable:
\begin{enumerate}[{\rm(a)}]
\item $\varphi\rightarrow\psi,\psi\rightarrow\chi\vdash\varphi\rightarrow\chi$,
\item $\vdash\varphi\rightarrow1$,
\item $\vdash\neg\neg\varphi\leftrightarrow\varphi$,
\item $\varphi\rightarrow\psi\vdash\neg\varphi\rightarrow\neg\psi\leftrightarrow\psi\rightarrow\varphi$.
\end{enumerate}
\end{theorem}

\begin{proof}
\
\begin{enumerate}[(a)]
\item This follows from (Sf) and (MP).
\item According to (B1),
\[
\vdash1\rightarrow(\varphi\rightarrow1)
\]
and according to (B4)
\[
\vdash0\rightarrow0,
\]
i.e.
\[
\vdash1.
\]
Thus applying (MP) we conclude
\[
\vdash\varphi\rightarrow1.
\]
\item This is (B3).
\item According to (R1) we have
\[
\varphi\rightarrow\psi\vdash(\neg\varphi\rightarrow\neg\psi)\rightarrow(\psi\rightarrow\varphi).
\]
Moreover,
\[
\vdash(\psi\rightarrow\varphi)\rightarrow(\neg\varphi\rightarrow\neg\psi)
\]
according to (Sf).
\end{enumerate}
\end{proof}

In what follows we show that this algebraic semantics is just that for algebras satisfying conditions (i) -- (viii) of Definition~\ref{def2}, i.e.\ for finite effect implication algebras. Similarly to the lattice case, the system $(\mathcal L,\vdash_{{\rm EA}})$ fulfills the properties of SIC. Again, applying Proposition~\ref{prop2}, we are going to show the following theorem:

\begin{theorem}
The equivalent algebraic semantics for $(\mathcal L,\vdash_{{\rm EA}})$ is axiomatized by the following identities and quasiidentities {\rm(}we abbreviate $x\rightarrow0$ by $x'$ and $0'$ by $1${\rm)}:
\begin{enumerate}[{\rm(1)}]
\item $x\rightarrow(y\rightarrow x)\approx1$,
\item $x\rightarrow x''\approx1$ and $x''\rightarrow x\approx1$,
\item $0\rightarrow x\approx1$,
\item $x\rightarrow x\approx1$,
\item $x=x\rightarrow y=1\Rightarrow y=1$,
\item $x\rightarrow y=1\Rightarrow(y\rightarrow z)\rightarrow(x\rightarrow z)=1$,
\item $x\rightarrow y=1\Rightarrow(x'\rightarrow y')\rightarrow(y\rightarrow x)=1$,
\item $x\rightarrow y=y\rightarrow x=1\Rightarrow x=y$,
\item $x\rightarrow y'=(x'\rightarrow y)\rightarrow z'=1\Rightarrow((x'\rightarrow y)'\rightarrow z)\rightarrow (x'\rightarrow (y'\rightarrow z))=1$,
\item $x\rightarrow y=1\Rightarrow(y\rightarrow x)\rightarrow x=y$.
\end{enumerate}
\noindent
This system just corresponds to properties {\rm(i)} -- {\rm(viii)} of Definition~\ref{def2}.
\end{theorem}

\begin{proof}
We will show that the corresponding algebra satisfies the conditions listed in Definition~\ref{def2}.
\begin{enumerate}[(i)]
\item The first two identities of (i) are just (3) and (4) and the remaining one follows from (4) and (1).
\item This is just (8).
\item Assume $x\rightarrow y=y\rightarrow z=1$. Then according to (6) we have $(y\rightarrow z)\rightarrow(x\rightarrow z)=1$ which according to (5) yields $x\rightarrow z=1$.

(i) -- (iii) show that the relation $\leq$ defined by $x\leq y$ if and only if $x\rightarrow y=1$ is a partial order relation with smallest element $0$ and greatest element $1$.
\item This follows from (6).
\item This follows from (2) and (8).
\item Assume $x\leq y$. Then $y'\leq x'$ according to (6) and
\[
y\rightarrow x=y''\rightarrow x''\leq x'\rightarrow y'\leq y\rightarrow x
\]
according to (v) and (7).
\item Assume  $x\rightarrow y'=(x'\rightarrow y)\rightarrow z'=1$, i.e. $x\leq y'$ and $x'\rightarrow y\leq z'$. We know by (1) that $y\leq x'\rightarrow y$, thus $y\leq z'$. Now, we can apply (7) and (v) to obtain $y'\rightarrow z=z'\rightarrow y$. Using (6) and (10), $x'\rightarrow y\leq z'$ yields $z'\rightarrow y\leq (x'\rightarrow y)\rightarrow y=x'$. This shows that $y'\rightarrow z\leq x'$ which due to (iv) gives $x\leq (y'\rightarrow z)'$. The converse implication follows by interchanging the elements $x,y,z$. Finally, under the above assumptions we have by (9) $(x'\rightarrow y)'\rightarrow z\leq x'\rightarrow (y'\rightarrow z)$. Again, by interchanging the elements $x,y,z$ we obtain the converse inequality and thus equality.
\item  This is just (1).
\end{enumerate}
\end{proof}

Due to the correspondence described in Theorem~\ref{th8}, our propositional logic $L_{{\rm EA}}$, i.e.\ system $\mathbf B$, is an algebraic axiomatization in Gentzen style of the logic of finite effect algebras which need not be lattice-ordered. A possible extension to effect algebras satisfying ACC could make problems since the last condition of Definition~\ref{def2} (mentioned after condition (viii)) equivalent to ACC cannot be easily described in our logic $L_{{\rm EA}}$.

Compliance with Ethical Standards: This study was funded by \"OAD, project CZ~02/2019, concerning the first and second author by IGA, project P\v rF~2020~014, and concerning the first and third author by the Austrian Science Fund (FWF), project I~4579-N, and the Czech Science Foundation (GA\v CR), project 20-09869L. The authors declare that they have no conflict of interest. This article does not contain any studies with human participants or animals performed by any of the authors.

Authors' addresses:

Ivan Chajda \\
Palack\'y University Olomouc \\
Faculty of Science \\
Department of Algebra and Geometry \\
17.\ listopadu 12 \\
771 46 Olomouc \\
Czech Republic \\
ivan.chajda@upol.cz

Radom\'ir Hala\v s \\
Palack\'y University Olomouc \\
Faculty of Science \\
Department of Algebra and Geometry \\
17.\ listopadu 12 \\
771 46 Olomouc \\
Czech Republic \\
radomir.halas@upol.cz

Helmut L\"anger \\
TU Wien \\
Faculty of Mathematics and Geoinformation \\
Institute of Discrete Mathematics and Geometry \\
Wiedner Hauptstra\ss e 8-10 \\
1040 Vienna \\
Austria, and \\
Palack\'y University Olomouc \\
Faculty of Science \\
Department of Algebra and Geometry \\
17.\ listopadu 12 \\
771 46 Olomouc \\
Czech Republic \\
helmut.laenger@tuwien.ac.at
\end{document}